\documentclass[reqno]{amsart}

\usepackage[T1]{fontenc}
\usepackage[utf8]{inputenc}
\usepackage{lmodern}

\usepackage[english]{babel}
\usepackage{amsmath,amssymb,amsthm,mathtools}
\usepackage{graphicx}
\usepackage{color}
\usepackage{hyperref}

\numberwithin{equation}{section}

\hypersetup{
	colorlinks=true,
	linkcolor=blue,
	citecolor=blue,
	urlcolor=blue,
	pdftitle={Broadcasting on Random Trees},
	pdfauthor={Tayfun Kardas, Vsevolod Shneer and Vitali Wachtel}
}

\theoremstyle{plain}
\newtheorem{theo}{Theorem}[section]

\newtheorem{lema}[theo]{Lemma}
\newtheorem{coro}[theo]{Corollary}

\theoremstyle{definition}
\newtheorem{defn}[theo]{Definition}

\theoremstyle{remark}

\newtheorem*{remark}{Remark}

\title[Broadcasting on Random Trees]
    {Broadcasting on Random Trees: Martingale Methods and Error Bounds}
    \thanks{
     The work of Tayfun Kardas was supported by the 
     Research Centre for Mathematical Modelling (RCM$^2$), Bielefeld University. 
     This financial support is gratefully acknowledged.
    }

\author[Kardas]{Tayfun Kardas}
\address{Faculty of Mathematics, Bielefeld University, Bielefeld, Germany}
\email{tkardas@math.uni-bielefeld.de}

\author[Shneer]{Vsevolod Shneer}
\address{Department of Actuarial Mathematics and Statistics, School of Mathematical and Computer Sciences, Heriot-Watt University, Edinburgh, UK}
\email{V.Shneer@hw.ac.uk}

\author[Wachtel]{Vitali Wachtel}
\address{Faculty of Mathematics, Bielefeld University, Bielefeld, Germany}
\email{wachtel@math.uni-bielefeld.de}



\begin{document}

\begin{abstract}
The broadcasting problem on trees concerns the transmission of binary information along the edges of a tree. In the formulation studied here, each vertex carries one of two colours, and along every edge the colour is transmitted with error probability $q$.
The central question is whether the root colour can be reconstructed from the observed colours of all vertices in the tree, but crucially, without knowing which vertex is the root.
We study the problem of reconstructing the root colour in a broadcasting process on
random trees, focusing on the majority rule estimator. We introduce a novel and adaptable
martingale approach for analysing the colour imbalance. This method provides a simpler
and more powerful alternative to existing techniques.
Our martingale approach unifies the analysis of the different tree models and leads to
linear upper bounds for the error probability of the majority rule estimator in each case.
This shows that, for a broad class of random tree models, the error probability of the
majority rule admits a linear upper bound in $q$.
\end{abstract}

\maketitle

\section{Introduction}

The broadcasting problem on trees concerns the transmission of a binary signal from the root along the edges of a tree, where each edge independently flips the transmitted bit with probability \(q\). The central question is whether the signal value at the root can be reconstructed from the observed tree with signal bits
(colours) assigned to every vertex. In the present paper, the underlying tree is a randomly growing recursive tree, and the location of the root is hidden from the observer. Thus, after the tree has been generated and coloured, one observes the entire tree and the colours of all vertices, and the goal is to reconstruct the colour of the root from this observation.


Let us first introduce the class of random trees we are going to study.
\begin{defn}[Unified recursive tree]\label{def:unified_recursive_tree}
	A {\it unified recursive tree} is a random recursive sequence of trees \((T_n)_{n\geq 1}\), where \(T_n\) has vertex set \([n]:=\{1,\dots,n\}\), $1$ denotes the root, and the probability that the vertex $n+1$ will be attached to a vertex $v$ is given by
	\[
	\mathbb P(n+1\sim v\mid T_n)=\frac{\alpha \deg_n^+(v)+\beta\,\mathbf 1_{\{v\neq 1\}}+1}{(\alpha+\beta)(n-1)+n},
	\qquad v\in[n],
	\]
	where \(\deg_n^+(v)\) denotes the out-degree of \(v\) in \(T_n\).
\end{defn}
We note here that when a new vertex attaches, we always think of the new edge as starting at an existing vertex (chosen according to the probabilities above) and ending at the new vertex. Hence, all vertices in the tree, apart from the root, have in-degree $1$ and the root has in-degree $0$. The out-degree of a vertex may take arbitrary  values.

Let us now describe the set of admissible values of parameters $\alpha$ and $\beta$. First, the denominator in the attachment probability should be positive. Therefore, we need to impose the condition $\alpha+\beta\ge-1$. If the out-degree of the root is $d_1$ then the numerator is equal to $\alpha d_1+1$. Here we have two possibilities: either $\alpha\in[0,\infty)$ or $\alpha=-\frac{1}{D_1}$ for some $D_1\in\{1,2,\ldots\}$. In the latter case one can attach at most $D_1$ vertices to the root. Let $d_2$ be the out-degree of a non-root vertex $v$. Then the numerator is equal to $\alpha d_2+\beta+1$. We have two possibilities to keep this number non-negative: $\alpha\in[0,\infty)$, $\beta\in[-1,\infty)$ or $\beta=-1-\alpha D_2$ for some $D_2\in\{1,2,\ldots\}$. Combining this with the restriction $\alpha+\beta\ge -1$, we conclude that the set of admissible values is given by
$$
\mathcal{D}:=\{(\alpha,\beta); \alpha\ge0,\beta\ge-1\text{ and }\alpha+\beta>-1\}
\cup\left\{\left(-\frac{1}{D_1},\frac{D_2}{D_1}-1\right);
D_1\ge1, D_2\ge 2\right\}.
$$

We have excluded all values $\alpha$ and $\beta$ such that $\alpha+\beta+1=0$, since we will need the condition $\alpha+\beta+1>0$ in our analysis.
Choosing \(\beta=0\) and \(\beta=\alpha\)
we get very simple increasing tree and shape exchangeable tree, respectively. These two classes have been considered in \cite{hartung2024random}.

Given the tree growth model, we now introduce the corresponding broadcasting process.
\begin{defn}[Broadcasting process]
	Let \((T_n)_{n\geq 1}\) be the tree growth model above. Each vertex carries one of two colours, which we identify with the values \(-1\) and \(1\). The root vertex \(1\) is assigned a random bit \(B_1\in\{-1,1\}\), chosen uniformly at random. For each \(n\geq 1\), once the new vertex \(n+1\) has chosen its parent \(p_{n+1}\), it inherits the bit of its parent with probability \(1-q\) and flips it with probability \(q\), independently of the tree growth and all previous broadcasting steps. Equivalently,
	\[
	B_{n+1}=B_{p_{n+1}}Z_{n+1},
	\]
	where \((Z_n)_{n\geq 2}\) are independent \(\{-1,1\}\)-valued random variables satisfying
	\[
	\mathbb P(Z_n=-1)=q,
	\qquad
	\mathbb P(Z_n=1)=1-q.
	\]
	The parameter \(q\in[0,1]\) is called the bit-flip probability.
\end{defn}
Although the broadcasting process is defined on the labelled recursive tree \(T_n\), only
the corresponding coloured unlabelled tree is observed. The reconstruction problem is
therefore to estimate the root bit \(B_1\) from this observation. In this paper, we study
the majority rule estimator.
\begin{defn}[Majority rule]
	Let \((T_n,(B_v)_{v\in[n]})\) be the broadcast-coloured tree of size \(n\), where
	\(B_v\in\{-1,1\}\) denotes the colour of vertex \(v\).
	The {\it majority rule} estimator is the \(\{-1,1\}\)-valued random variable
	\[
	\widehat B_{\mathrm{maj},n}
	:=
	\begin{cases}
		1, & \sum_{v=1}^n B_v > 0,\\
		-1, & \sum_{v=1}^n B_v < 0,\\
		\xi_n, & \sum_{v=1}^n B_v = 0,
	\end{cases}
	\]
	where \(\xi_n\) is independent of the broadcasting process and satisfies
	\[
	\mathbb P(\xi_n=1)=\mathbb P(\xi_n=-1)=\frac12.
	\]
	Its error probability is defined by
	\[
	R_{\mathrm{maj}}(n,q):=\mathbb P\!\left(\widehat B_{\mathrm{maj},n}\neq B_1\right),
	\qquad
	R_{\mathrm{maj}}(q):=\limsup_{n\to\infty} R_{\mathrm{maj}}(n,q).
	\]
\end{defn}

The broadcasting problem on random recursive trees has been introduced and studied by Addario-Berry, Devroye, Lugosi and Velona \cite{addario2022broadcasting}, who have analysed the quality of the majority rule for uniform recursive trees ($\alpha=\beta=0$) and for the linear preferential model (this corresponds to our model with $\alpha>0$, $\beta=0$) of random trees. Their proofs are based on a rather careful analysis of underlying trees, and they are different in two models considered there.

More recently, Althaus, Hartung and Steiner \cite{hartung2024random} have suggested a supermartingale approach
to the broadcasting problem, which is in some sense 'nonparametric': it does not depend on a particular growing mechanism of the tree. They have shown that
$R_{\mathrm{maj}}(q)=O(\sqrt{q})$, which is weaker than the linear in $q$ bounds obtained in \cite{addario2022broadcasting}.

The purpose of the present paper is to develop a martingale approach that yields linear upper bounds for the error probability of the majority rule for unified recursive trees.

	The statement below is our main result.
\begin{theo}\label{thm:main}
	Assume that $q<1/2$ and that an admissible pair
    $(\alpha,\beta)$ satisfies 
    $$
    a:=1-\frac{2q(\beta+1)}{\alpha+\beta+1}>\frac{1}{2}.
    $$
    Then
	\begin{align}\label{eq:main_bound}
	R_{\mathrm{maj}}(q)\le \frac{\Gamma^2\left(\frac{1}{\alpha+\beta+1}+a\right)}{\Gamma\left(\frac{1}{\alpha+\beta+1}\right)\Gamma\left(\frac{1}{\alpha+\beta+1}+2a\right)}\left(\alpha+1+\frac{\beta+1}{2a-1}\right)-1.
	\end{align}
\end{theo}
We note that the condition $a > 1/2$ is not restrictive. It is easy to see that it holds, in particular, for all $q<1/8$. Indeed, note first that the condition is equivalent to the requirement that $4q(\beta+1) < \alpha+\beta+1$. Note that for all admissible values of $\alpha$ and $\beta$, we have $\beta+1 \ge 0$. Hence, the condition is satisfied for $q<1/4$ for all $\alpha\ge 0$.  If $\alpha < 0$, we have that $\alpha = -1/D_1$ and $\beta = D_2/D_1 - 1$ for some $D_1\ge1$ and $D_2\ge 2$. Therefore, $4q(\beta+1) < \alpha+\beta+1$ is equivalent to
$$
4q\frac{D_2}{D_1} < \frac{D_2}{D_1}-\frac{1}{D_1}.
$$
The assumption $D_2\ge2$ implies that this condition is valid for all $q<1/8$.

It is easy to see that if we take $q=0$ ($a=1$) then the right hand side in \eqref{eq:main_bound} becomes zero. Furthermore, one can show that the right hand side grows linearly in $q$ for sufficiently small $q$. This fact implies that $R_{\mathrm{maj}}(q)\le C_{\alpha,\beta} q$ for all $q$.

Let us now implement this in the particular case of uniform recursive trees.
\begin{coro}\label{cor:main}
	For the uniform random recursive tree ($\alpha=\beta=0$)
    one has
	\[
	R_{\mathrm{maj}}(q)\le 
    \frac{\Gamma^2(2-2q)}{\Gamma(3-4q)}
    \left(1+\frac{1}{1-4q}\right)-1 \quad\text{for all } q\in[0,1/4).
	\]
    Moreover,
    $$
    R_{\mathrm{maj}}(q)
    \le 8\left(\frac{9\Gamma^2(3/4)}{4\sqrt{\pi}}-1\right)q
    \le 7.25 q \quad\text{for all } q\in[0,1/8)
    $$
    and
    $$
    R_{\mathrm{maj}}(q)\le 8 q \quad\text{for all } q\in[0,1].
    $$
\end{coro}
By the Taylor formula, as $q\downarrow0$,
$$
\frac{\Gamma^2(2-2q)}{\Gamma(3-4q)}\left(1+\frac{1}{1-4q}\right)-1
=2\left(1+\Gamma'(3)-2\Gamma'(2)\right)q+O(q^2).
$$
Differentiating the equality $\Gamma(x+1)=x\Gamma(x)$, one gets
$\Gamma'(x+1)=\Gamma(x)+x\Gamma'(x)$ and, in particular,
$\Gamma'(3)=1+2\Gamma'(2)$. Consequently,
$$
\frac{\Gamma^2(2-2q)}{\Gamma(3-4q)}\left(1+\frac{1}{1-4q}\right)-1
=4q+O(q^2).
$$
Thus, our simplified linear bound $7.25 q$ is not too crude. The expansion above also provides a more precise bound for small $q$. 

\section{Random walk representation and martingales}
Let \(\Delta_1(n)\) denote the colour imbalance in \(T_n\), where we identify the value \(1\) with red and the value \(-1\) with blue. Let \(R(n)\) and \(B(n)\) denote, respectively, the numbers of red and blue vertices in \(T_n\). Thus,
\[
\Delta_1(n):=R(n)-B(n).
\]
Since the majority rule depends only on the sign of this quantity, the precise tree structure is not directly relevant at this stage. By symmetry, we may assume without loss of generality that the root bit is red. Moreover, since the majority rule makes an unbiased random choice in the case of a tie,
\[	
R_{\mathrm{maj}}(n,q)= \mathbb P\left(\Delta_1(n)<0\right)+\frac12 \mathbb P\left(\Delta_1(n)=0\right)\le \mathbb P\left(\Delta_1(n)\le 0\right).
\]
Hence,
\begin{equation}\label{eq:Rmaj-Delta1}
	R_{\mathrm{maj}}(q)\le \limsup_{n\to\infty}\mathbb P\left(\Delta_1(n)\le 0\right).
\end{equation}
Therefore, in order to prove Theorem~\ref{thm:main} and Corollary~\ref{cor:main}, it suffices to derive suitable upper bounds on the right hand side of \eqref{eq:Rmaj-Delta1}.

As we shall see later, in the case of uniform random trees, \(\alpha=\beta=0\), the process \(\Delta_1(n)\) alone is sufficient to describe the evolution of the broadcasting process. In the general case, however, one needs additional information to determine the attachment probabilities. Let \(C_R(n)\) and \(C_B(n)\) denote, respectively, the numbers of children of red and blue vertices in the tree \(T_n\). We then set
\[
\Delta_2(n):=C_R(n)-C_B(n).
\]
As a result we have a $2$-dimensional process 
$$
\Delta(n)=(\Delta_1(n),\Delta_2(n)),\quad n\ge1
$$
with the initial condition 
$$
\Delta_1(1)=1,\qquad \Delta_2(1)=0,
$$
corresponding to the single red root.

Let $D(n)$ denote the increments of the process $\Delta(n)$, that is,
\[
D(n):=\Delta(n+1)-\Delta(n)=(D_1(n),D_2(n)) \in \{-1,1\}^2.
\]

The idea to introduce $\Delta_2(n)$ has been suggested in \cite{hartung2024random}. Also the next result is a slight generalisation of Lemma 2.8 in this paper.
\begin{lema}\label{lem:red-attachment-probability}
Let \(\{\mathcal{F}_n\}\) denote the natural filtration generated by the sequence $\{\Delta(n)\}$. Then
	\begin{align}\label{eq:p_vs_se}
    \nonumber
		p(\alpha,\beta,n)
        &:=\mathbb{P}(n+1 \sim \text{red vertex} \mid \mathcal{F}_n)\\
        &=\frac12\left(1+\frac{(\beta+1)\Delta_1(n)+\alpha\Delta_2(n)-\beta}{(\alpha+\beta)(n-1)+n}\right).
	\end{align}
\end{lema}
\begin{proof}
	It turns out that the conditional probability that the new vertex \(n+1\) attaches to a blue vertex is easier to compute.
	Since the root is red, \(\mathbf 1_{\{v\neq 1\}}=1\) for every blue vertex $v$ in the tree $T_n$.
    Consequently,
    \begin{align*}
    	\mathbb{P}(n+1 \sim \text{blue vertex} \mid \mathcal{F}_n)
    	&= \sum_{\substack{v \in [n] \\ B_v = -1}}
    	\frac{\alpha \deg_n^+(v)+\beta\mathbf 1_{\{v\neq 1\}}+1}
    	{(\alpha+\beta)(n-1)+n} \\
    	&= \frac{\alpha C_B(n)+(\beta+1)B(n)}
    	{(\alpha+\beta)(n-1)+n}.
    \end{align*}
    It is immediate from the definitions of $B(n)$, $R(n)$, $C_B(n)$ and $C_R(n)$ that
    $$
    B(n)+R(n)=n\quad\text{and}\quad C_B(n)+C_R(n)=n-1.
    $$
    These equalities imply that 
    $$
    B(n)=\frac{n-\Delta_1(n)}{2}
    \quad\text{and}\quad 
    C_B(n)=\frac{n-1-\Delta_2(n)}{2}.
    $$
	Therefore,
	\begin{align*}
		\mathbb{P}(n+1 \sim\text{blue vertex}\mid \mathcal F_n)
		&=\frac12 \frac{\alpha(n-1-\Delta_2(n))+(\beta+1)(n-\Delta_1(n))}{(\alpha+\beta)(n-1)+n}\\
		&=\frac12\left(1-\frac{(\beta+1)\Delta_1(n)+\alpha\Delta_2(n)-\beta}{(\alpha+\beta)(n-1)+n}\right).
	\end{align*}
	Taking the complement, we obtain \eqref{eq:p_vs_se}.
\end{proof}

The conditional probabilities of \(D_1(n)\) are then given by
\begin{align*}
	\mathbb{P}\left(D_1(n)=1\mid\mathcal{F}_n\right)
	&=p(\alpha,\beta,n)(1-q)+(1-p(\alpha,\beta,n))q\\
	&=q+p(\alpha,\beta,n)(1-2q)
\end{align*}
and
\begin{equation*}
	\mathbb{P}\left(D_1(n)=-1\mid\mathcal{F}_n\right)=1-q-p(\alpha,\beta,n)(1-2q).
\end{equation*}
Hence,
\begin{align}
	\mathbb{E}[D_1(n)\mid\mathcal{F}_n]
	&=q+p(\alpha,\beta,n)(1-2q)-\left(1-q-p(\alpha,\beta,n)(1-2q)\right)\nonumber\\
	&=2q-1+2p(\alpha,\beta,n)(1-2q)\nonumber\\
	&=(1-2q)(2p(\alpha,\beta,n)-1). \label{eq:D1-cond-exp}
\end{align}
For \(D_2(n)\) we have
\begin{equation*}
	\mathbb{P}\left(D_2(n)=1\mid\mathcal{F}_n\right)=p(\alpha,\beta,n)
\end{equation*}
and therefore,
\begin{equation*}
	\mathbb{P}\left(D_2(n)=-1\mid\mathcal{F}_n\right)=1-p(\alpha,\beta,n).
\end{equation*}
Hence,
\begin{equation}
	\mathbb{E}[D_2(n)\mid\mathcal{F}_n]=2p(\alpha,\beta,n)-1. \label{eq:D2-cond-exp}
\end{equation}
These equalities imply that the sequence $\{\Delta(n)\}$ is a $2$-dimensional time-inhomogeneous Markov chain. To analyse the asymptotic properties of this chain we shall use martingales which we introduce now.

Combining \eqref{eq:D1-cond-exp} and \eqref{eq:D2-cond-exp}, we obtain
\begin{align*}
	\mathbb{E}\left[D_1(n)-(1-2q)D_2(n)\mid\mathcal{F}_n\right]
	&=(1-2q)(2p(\alpha,\beta,n)-1)-(1-2q)(2p(\alpha,\beta,n)-1)\\
	&=0.
\end{align*}
Consequently, the process \(\{L(n)\}_{n\ge1}\) defined by
\begin{equation*}
	L(n):=\Delta_1(n)-\gamma\,\Delta_2(n)
\end{equation*}
with \(\gamma:=1-2q\) is a martingale.

To introduce the second martingale we define
\begin{equation*}
	S(n):=(\beta+1)\Delta_1(n)+\alpha\Delta_2(n)-\beta, \qquad n\geq 1.
\end{equation*}

By \eqref{eq:D1-cond-exp} and \eqref{eq:D2-cond-exp}, we have
\begin{align*}
	\mathbb{E}[S(n+1)\mid\mathcal{F}_n]
	&= S(n)+(\beta+1)\mathbb{E}[D_1(n)\mid\mathcal{F}_n]+\alpha\mathbb{E}[D_2(n)\mid\mathcal{F}_n] \\
	&= S(n)+(\beta+1)(1-2q)\bigl(2p(\alpha,\beta,n)-1\bigr)+\alpha\bigl(2p(\alpha,\beta,n)-1\bigr)\\
	&=S(n)+\left((\beta+1)(1-2q)+\alpha\right)\left(2p(\alpha,\beta,n)-1\right).
\end{align*}
Furthermore, the equality \eqref{eq:p_vs_se} can be written as
\[
2p(\alpha,\beta,n)-1=\frac{S(n)}{(\alpha+\beta)(n-1)+n}\;.
\]
Hence,
\begin{equation}\label{eq:Sn-recurrence}
	\mathbb{E}[S(n+1)\mid\mathcal{F}_n]=S(n)\left(1+\frac{\theta}{nZ_{\alpha,\beta}(n)}\right),
\end{equation}
where
\begin{equation}\label{eq:def-theta-Z}
	\theta:= (\beta+1)(1-2q)+\alpha\quad\text{and}\quad Z_{\alpha,\beta}(n):=1+(\alpha+\beta)\left(1-\frac{1}{n}\right).
\end{equation}
Setting now
\begin{equation}\label{eq:cn_product}
	c_n=\prod_{k=1}^{n-1}\left(1+\frac{\theta}{kZ_{\alpha,\beta}(k)}\right)^{-1},\quad n\ge1,
\end{equation}
we infer from \eqref{eq:Sn-recurrence} that the sequence
\begin{equation}
\label{eq:Mn-def}
M(n):=c_nS(n),\quad n\ge1
\end{equation}
is a martingale. This martingale is very similar to the standard martingale for Elephant's random walk, see \cite{Bercu18}.
\begin{remark}
In the case when $\alpha=0$, the attaching probability in \eqref{eq:p_vs_se}  and the martingale $M(n)$ do not depend on $\Delta_2(n)$. This implies that to analyse the majority rule in the case $\alpha=0$ it is sufficient to consider $\Delta_1(n)$ only.  
\end{remark}

\section{Central Limit Theorem for \(L(n)\).}
\label{sec:clt-martingale}
In this short section we show that the martingale $L(n)$ satisfies the central limit theorem.

Let $\{\delta(n)\}$ denote the increments of \(\{L(n)\}\), that is,
\[
\delta(n)=L(n+1)-L(n)=D_1(n)-\gamma D_2(n),\quad n\ge1.
\]
To compute the conditional second moment of $\delta(n)$, we write
\[
\mathbb{E}[\delta^2(n)\mid\mathcal{F}_n]
=\mathbb{E}[D_1^2(n)\mid\mathcal{F}_n]
+\gamma^2\mathbb{E}[D_2^2(n)\mid\mathcal{F}_n]
-2\gamma\,\mathbb{E}[D_1(n)D_2(n)\mid\mathcal{F}_n].
\]
Because \(D_1^2(n)=D_2^2(n)=1\), the first two terms equal \(1+\gamma^2\).
For the last term, note that \(D_1(n)D_2(n)=1\) exactly when the new vertex receives the
same colour as its parent (which happens with probability \(1-q\)), and \(D_1(n)D_2(n)=-1\)
otherwise. Hence, independently of the parent's colour,
\begin{equation}\label{eq:D1D2-cross}
	\mathbb{E}[D_1(n)D_2(n)\mid\mathcal{F}_n]=1\cdot(1-q)+(-1)\cdot q=1-2q.	
\end{equation}
Recalling that \(\gamma=1-2q\), we finally obtain
\[
\mathbb{E}[\delta^2(n)\mid\mathcal{F}_n]=1+\gamma^2-2\gamma^2=1-\gamma^2=1-(1-2q)^2=4q(1-q).
\]
Note also that the increments $\delta(k)$ are uniformly bounded:
\[
|\delta(n)|
\le |D_1(n)|+|\gamma|\,|D_2(n)|
\le 1+|\gamma|
\le 2.
\]
This implies that the Lindeberg condition is valid and, owing to the central limit theorem for martingales, see  \cite[Theorem~35.12]{Billingsley1995}, we conclude that
\[
\frac{L(n)-L(1)}{\sqrt{4q(1-q)n}}
=\frac{\sum_{k=1}^{n-1}\delta(k)}{\sqrt{4q(1-q)n}}\xrightarrow{d}\mathcal{N}(0,1).
\]
Recalling that \(L(1)=1\) we conclude that
\begin{equation}\label{eq:clt-L}
	\frac{L(n)}{\sqrt n}\xrightarrow{d}\mathcal N\left(0,4q(1-q)\right)
\end{equation}
for all \(q\in (0,1)\).

In the subsequent section we shall show that $S(n)$ grows as $n^a$
with some $a>1/2$. The convergence \eqref{eq:clt-L} implies then that $\Delta_1(n)$, $\Delta_2(n)$ grow also as $n^a$. From the technical point of view, \eqref{eq:clt-L} gives one a more efficient information on the relation between $\Delta_1(n)$ and $M(n)$ than Lemma 4.6 in \cite{hartung2024random}. Also the proof of \eqref{eq:clt-L} is much simpler than the arguments used in \cite{hartung2024random}.
\section{Asymptotic properties of $M(n)$}
\label{sec:martingale-unified}
We start by determining the asymptotics for the normalising sequence $\{c_n\}$. Before we formulate the next lemma, note that, under the assumption $a > 1/2$ of Theorem \ref{thm:main}, we have that $\theta$ defined in \eqref{eq:def-theta-Z} is positive. Indeed, the inequality $\theta>0$ is equivalent to $2q(\beta+1)<\alpha+\beta+1$ and it holds as the assumption $a>1/2$ is equivalent to $4q(\beta+1) < \alpha+\beta+1$.

\begin{lema}
\label{lem:Cn-asymp}
Under the assumptions of Theorem \ref{thm:main},
\begin{equation}\label{eq:c_n-asymptotic}
c_n\sim Q n^{-a}
\quad\text{as }n\to\infty,
\end{equation}
where
$$
a=1-\frac{2q(\beta+1)}{\alpha+\beta+1}
\quad\text{and}\quad 
Q=\frac{\Gamma(\frac{1+\theta}{1+\alpha+\beta})}
{\Gamma(\frac{1}{1+\alpha+\beta})}.
$$
\end{lema}
\begin{proof}
By \eqref{eq:def-theta-Z}, we have
\begin{align*}
	1+\frac{\theta}{kZ_{\alpha,\beta}(k)}
	&=1+\frac{\theta}{k+(\alpha+\beta)(k-1)} \\
	&=\frac{(\alpha+\beta+1)k-(\alpha+\beta)+\theta}{(\alpha+\beta+1)k-(\alpha+\beta)} \\
	&=\left(1+\frac{\theta-(\alpha+\beta)}{(\alpha+\beta+1)k}\right)\left(1-\frac{\alpha+\beta}{(\alpha+\beta+1)k}\right)^{-1}.
\end{align*}
Substituting this into \eqref{eq:cn_product} yields
\[
c_n=\prod_{k=1}^{n-1}\left(1+\frac{\theta-(\alpha+\beta)}{(\alpha+\beta+1)k}\right)^{-1}\; \prod_{k=1}^{n-1}\left(1-\frac{\alpha+\beta}{(\alpha+\beta+1)k}\right).
\]
The two products can now be studied separately. Recall that for every constant
\(\kappa>-1\),
\begin{equation*}
	\prod_{k=1}^{n}\left(1+\frac{\kappa}{k}\right)=\frac{\Gamma(n+1+\kappa)}{\Gamma(1+\kappa)\Gamma(n+1)},
\end{equation*}
and hence, by Stirling's formula,
\begin{equation}\label{eq:product-asymptotic}
	\prod_{k=1}^{n}\left(1+\frac{\kappa}{k}\right)\sim \frac{n^\kappa}{\Gamma(1+\kappa)}\qquad\text{as } n\to\infty.
\end{equation}
Using \eqref{eq:product-asymptotic} with
\[
\kappa_1 = \frac{\theta - (\alpha+\beta)}{\alpha+\beta+1}
\quad\text{and}\quad
\kappa_2 = -\frac{\alpha+\beta}{\alpha+\beta+1},
\]
we conclude that 
$$
c_n\sim
\frac{\Gamma(1+\kappa_1)}{\Gamma(1+\kappa_2)}n^{\kappa_2-\kappa_1}
\quad\text{as }n\to\infty.
$$
Noting that 
$$
\kappa_2-\kappa_1=-1+\frac{2q(\beta+1)}{\alpha+\beta+1}
$$
and
$$
1+\kappa_1=\frac{1+\theta}{1+\alpha+\beta},
\quad
1+\kappa_2=\frac{1}{1+\alpha+\beta},
$$
we get the desired relation.
\end{proof}
To show that the martingale \(M(n)=c_nS(n)\) converges in $L^2$ we derive next an upper bound for its second moments.
\begin{lema}
\label{lem:cond2}
	For every $n\ge1$ we have
	\begin{equation}\label{eq:M(n)-second-moment}
		\mathbb{E}\left[M^2(n+1) \mid \mathcal{F}_n\right] =\frac{c_{n+1}^2}{c_n^2}\left(1+\frac{2\theta}{n\,Z_{\alpha,\beta}(n)}\right) M^2(n) + c_{n+1}^2 K, 
	\end{equation}
	where
	\[
	K := (\beta+1)^2+\alpha^2+2\alpha(\beta+1)\gamma.
	\]
\end{lema}
\begin{proof}
	Set \(U(n):=S(n+1)-S(n)\). It follows then from the definition of $S(n)$ that
	\[
	U(n)=(\beta+1)D_1(n)+\alpha D_2(n).
	\]
	According to \eqref{eq:Sn-recurrence}, 
	\begin{equation}\label{eq:U(n)-first-moment}
		\mathbb{E}[U(n)\mid\mathcal{F}_n] = \frac{\theta S(n)}{nZ_{\alpha,\beta}(n)}\, .
	\end{equation}
	By \eqref{eq:D1D2-cross} and \(D_1^2(n) = D_2^2(n) = 1\), we also have
	\begin{equation*}
		\mathbb{E}[U^2(n)\mid\mathcal{F}_n]
		= (\beta+1)^2+\alpha^2+2\alpha(\beta+1)\gamma
		= K.
	\end{equation*}
	This together with \eqref{eq:U(n)-first-moment} yields
	\begin{align}
		\mathbb{E}[S^2(n+1)\mid\mathcal{F}_n]
		&=\mathbb{E}[(S(n)+U(n))^2\mid\mathcal{F}_n] \nonumber\\
		&= S^2(n) + 2S(n)\mathbb{E}[U(n)\mid\mathcal{F}_n] + \mathbb{E}[U^2(n)\mid\mathcal{F}_n] \nonumber\\
		&= S^2(n)\left(1+\frac{2\theta}{nZ_{\alpha,\beta}(n)}\right) + K.\label{eq:S(n)-second-moment}
	\end{align}
	Multiplying by \(c_{n+1}^2\) and using \(S(n) = \frac{M(n)}{c_n}\) gives \eqref{eq:M(n)-second-moment}.
\end{proof}
\begin{lema}
\label{lem:Mn-conv}
Assume that $\alpha,\beta$ and $q$ are such that
$$
a=1-\frac{2q(\beta+1)}{\alpha+\beta+1}>\frac{1}{2}.
$$
Then there exists a random variable $M(\infty)$ such that
$$
M(n)\to M(\infty)
$$
almost surely and in $L^2$.
\end{lema}
\begin{proof}
It is immediate from the definition of $c_n$ that 
$$
\frac{c_{n+1}}{c_n}
=\frac{1}{1+\frac{\theta}{nZ_{\alpha,\beta}(n)}}.
$$
Therefore,
$$
\frac{c_{n+1}^2}{c_n^2}
\left(1+\frac{2\theta}{nZ_{\alpha,\beta}(n)}\right)
=\frac{1+\frac{2\theta}{nZ_{\alpha,\beta}(n)}}{
1+\frac{2\theta}{nZ_{\alpha,\beta}(n)}+
\frac{\theta^2}{n^2Z_{\alpha,\beta}^2(n)}}\le 1.
$$
Applying this inequality to the right hand side of 
\eqref{eq:M(n)-second-moment} and taking expectations, we obtain
\[
\mathbb{E}[M^2(n+1)]\le \mathbb{E}[M^2(n)]+Kc_{n+1}^2.
\]
Iterating this inequality and recalling that $M(1)=1$, we conclude that 
\[
\mathbb{E}[M^2(n)]\le 1+K\sum_{k=2}^{n}c_k^2,\quad n\ge1.
\]
The assumption $a>1/2$ and \eqref{eq:c_n-asymptotic} imply that
the sequence $\{c_n^2\}$ is summable.
Therefore,
\[
\sup_{n\ge1}\mathbb{E}[M^2(n)]<\infty.
\]
The martingale convergence theorem yields the desired result.
\end{proof}
In the proof of our main result we will need an exact expression for the variance of the limiting random variable $M(\infty)$.
\begin{lema}
\label{lem:var-asymp}    
Under the conditions of the previous lemma one has
$$
\mathbb{E}[M^2(\infty)]=Q^2\frac{\Gamma(1+\kappa_2)}{\Gamma(1+\kappa_2+2a)}
\left(1+K\frac{1+\kappa_2}{2a-1}\right).
$$
\end{lema}
\begin{proof}
Set
$$
s_n:=\mathbb{E}[S^2(n)]
\quad\text{and}\quad
A_n:=1+\frac{2\theta}{nZ_{\alpha,\beta}(n)}.
$$
It follows then from \eqref{eq:S(n)-second-moment} that
\begin{equation*}
s_{n+1}=A_ns_n+K,\quad n\ge1.
\end{equation*}
Iterating this equation and noting that $s_1=1$, one gets easily 
\begin{equation}
\label{eq:sn-soluton}
s_n=\left(1+K\sum_{j=2}^n\prod_{k=1}^{j-1}A_k^{-1}\right)
\prod_{k=1}^{n-1}A_k.
\end{equation}
Recall the definitions of $\kappa_2$ and $a$ in 
Lemma~\ref{lem:Cn-asymp}. One has the following representation for $A_n$:
\[
A_n
=1+\frac{2\theta}{(\alpha+\beta+1)n-(\alpha+\beta)}
=\frac{(\alpha+\beta+1)n+2\theta-(\alpha+\beta)}{(\alpha+\beta+1)n-(\alpha+\beta)}
=\frac{n+(\kappa_2+2a)}{n+\kappa_2}\; .
\]
Consequently,
\begin{equation}\label{eq:An-form}
\prod_{k=1}^{n-1}A_k=
\frac{\Gamma(n+\kappa_2+2a)}{\Gamma(1+\kappa_2+2a)}
\frac{\Gamma(1+\kappa_2)}{\Gamma(n+\kappa_2)}.
\end{equation}
Applying now Stirling's formula, we conclude that
\begin{equation}
\label{eq:An-asymp}
\prod_{k=1}^{n-1}A_k\sim 
\frac{\Gamma(1+\kappa_2)}{\Gamma(1+\kappa_2+2a)}n^{2a}
\quad\text{as }n\to\infty.
\end{equation}
Using this asymptotic relation and recalling that $a>1/2$, we infer that $\sum_{j=2}^\infty \prod_{k=1}^{j-1}A_k^{-1}$ is finite. We next determine the exact value of this series.
Due to \eqref{eq:An-form},
\begin{align*}
\sum_{j=2}^\infty \prod_{k=1}^{j-1}A_k^{-1}
&=\frac{\Gamma(1+\kappa_2+2a)}{\Gamma(1+\kappa_2)}
\sum_{j=2}^\infty
\frac{\Gamma(j+\kappa_2)}{\Gamma(j+\kappa_2+2a)}\\
&=\frac{\Gamma(1+\kappa_2+2a)}{\Gamma(1+\kappa_2)\Gamma(2a)}
\sum_{j=2}^\infty B(j+\kappa_2,2a),
\end{align*}
where $B(x,y)$ stands for the beta function. Recalling that 
$$
B(x,y)=\int_0^1t^{x-1}(1-t)^{y-1}dt
$$
and using Tonelli's theorem, we obtain 
\begin{align*}
\sum_{j=2}^\infty B(j+\kappa_2,2a)
&=\int_0^1\left(\sum_{j=2}^\infty t^{j+\kappa_2-1}\right)
(1-t)^{2a-1}dt\\
&=\int_0^1t^{1+\kappa_2}(1-t)^{2a-2}dt
=B(2+\kappa_2,2a-1).
\end{align*}
Consequently,
\begin{align*}
\sum_{j=2}^\infty \prod_{k=1}^{j-1}A_k^{-1}
=\frac{1+\kappa_2}{2a-1}.
\end{align*}
Combining this with \eqref{eq:An-asymp}, we conclude that
$$
s_n\sim 
\frac{\Gamma(1+\kappa_2)}{\Gamma(1+\kappa_2+2a)}
\left(1+K\frac{1+\kappa_2}{2a-1}\right)
n^{2a}
\quad\text{as }n\to\infty.
$$
Recalling that $M(n)=c_nS(n)$ and using Lemma~\ref{lem:Cn-asymp}, we obtain
\begin{align*}
\lim_{n\to\infty}\mathbb{E}[M^2(n)]
=\lim_{n\to\infty}c_n^2s_n
=Q^2\frac{\Gamma(1+\kappa_2)}{\Gamma(1+\kappa_2+2a)}
\left(1+K\frac{1+\kappa_2}{2a-1}\right).
\end{align*}
Recalling that $M(n)$ converges to $M(\infty)$ in $L^2$,we finish the proof.
\end{proof}
\section{Proof of the main result}
Recall that
\[
(\beta+1)\Delta_1(n)+\alpha\Delta_2(n)=S(n)+\beta
\quad\text{and}\quad
\Delta_1(n)-\gamma\Delta_2(n)=L(n).
\]
Therefore,
\begin{equation}\label{eq:Delta_1-S(n)-L(n)}
	\Delta_1(n)=\frac{\gamma(S(n)+\beta)+\alpha L(n)}{\theta},
	\quad\text{and}\quad
	\Delta_2(n)=\frac{(S(n)+\beta)-(\beta+1)L(n)}{\theta}.
\end{equation}

Fix some $\varepsilon\in(0,\frac{\gamma}{2\theta})$ and some $A>0$.
We then have
\begin{align}
	\mathbb{P}\left(\Delta_1(n)\le 0\right)
	&\le \mathbb{P}\left(c_n\Delta_1(n)\le \varepsilon\right) \nonumber\\
	&= \mathbb{P}\left(c_n\Delta_1(n)\le\varepsilon,\; |L(n)|\le A\sqrt{n}\right)+\mathbb{P}\left(c_n\Delta_1(n)\le\varepsilon,\; |L(n)|> A\sqrt{n}\right) \nonumber\\
	&=: P_1(n) + P_2(n). \label{eq:P_1-P_2}
\end{align}

Owing to \eqref{eq:clt-L}, we conclude that, uniformly in $q\in(0,1)$, 
\begin{align}
	\limsup_{n\to\infty} P_2(n)
	&\le \limsup_{n\to\infty}\mathbb P\left(|L(n)|>A\sqrt n\right) \nonumber\\
	&= 2\left(1-\Phi\left(\frac{A}{2\sqrt{q(1-q)}}\right)\right) \le 2\left(1-\Phi(A)\right), \label{eq:P_2-bound}
\end{align}
where \(\Phi\) denotes the distribution function of the standard normal law.

To derive an upper bound for $P_1(n)$, we notice first that, due to \eqref{eq:Delta_1-S(n)-L(n)},
\begin{align*}
P_1(n)
&=\mathbb{P}(c_n(\gamma(S(n)+\beta)+\alpha L(n))\le \theta\varepsilon,|L(n)|\le A\sqrt{n})\\
&\le \mathbb{P}\left(M(n)\le\frac{\theta\varepsilon+|\alpha| Ac_n\sqrt{n}-\beta\gamma c_n}{\gamma}\right).
\end{align*}
Since $a>1/2$, $c_n\sqrt{n}$ converges to zero. Thus, for all sufficiently large $n$, 
$$
P_1(n)\le\mathbb{P}(M(n)\le 2\theta\varepsilon/\gamma).
$$
Applying now the Chebyshev inequality and recalling that $\mathbb{E}[M(n)]=1$, we conclude that
$$
P_1(n)\le \frac{\mathbb{E}[M^2(n)]-1}{(1-2\theta\varepsilon/\gamma)^2}.
$$
Letting here $n\to\infty$, we finally arrive at
\begin{equation*}
\limsup_{n\to\infty}P_1(n)\le 
\frac{\mathbb{E}[M^2(\infty)]-1}{(1-2\theta\varepsilon/\gamma)^2}.
\end{equation*}

Combining this with \eqref{eq:P_1-P_2} and \eqref{eq:P_2-bound}, we obtain
\[
\limsup_{n\to\infty}\mathbb{P}\left(\Delta_1(n)\le 0\right)
\le \frac{\mathbb{E}[M^2(\infty)]-1}{(1-2\theta\varepsilon/\gamma)^2}
+ 2\left(1-\Phi(A)\right).
\]
Letting first \(A\to\infty\) and then $\varepsilon\to0$,
we infer that
\begin{equation}\label{eq:Var(M_infty)}
	\limsup_{n\to\infty}\mathbb{P}\left(\Delta_1(n)\le 0\right)
	\le \mathbb{E}[M^2(\infty)]-1 .
\end{equation}
Combining this with Lemma~\ref{lem:var-asymp}, we finish the proof of the theorem.
\section{Linear in $q$ bound for uniform recursive trees.}

\label{sec:uniform-case}
In this section we consider uniform recursive trees and prove Corollary~\ref{cor:main}. Letting $\alpha=\beta=0$ in Theorem~\ref{thm:main}, we conclude that, for $q\in(0,1/4)$,
$$
R_{\mathrm{maj}}(q)\le \frac{\Gamma(2-2q)^2}{\Gamma(3-4q)}\cdot\frac{2-4q}{1-4q}-1.
$$
Define
\begin{equation*}
	f(q)=\frac{\Gamma(2-2q)^2}{\Gamma(3-4q)}\cdot\frac{2-4q}{1-4q}-1.
\end{equation*}
\begin{lema}\label{lem:F-8q}
	For every \(q\in[0,\tfrac18]\), we have
\begin{equation}\label{eq:F-linear-8}
	f(q)\le 7.25q. 
\end{equation}
\end{lema}

\begin{proof}
	Recall that \(\gamma=1-2q\). Then the assumption $q\in[0,\tfrac18]$ is equivalent to \(\gamma\in[\tfrac34,1]\). Moreover,
	\begin{align}\label{eq:F-gamma}
    \nonumber
		f(q)&=\frac{\Gamma^2(\gamma+1)}  
            {\Gamma(2\gamma+1)}\cdot\frac{2\gamma}{2\gamma-1}-1\\
            &=\frac{\Gamma^2(\gamma+1)}  
            {\Gamma(2\gamma)}\cdot\frac{1}{2\gamma-1}-1=:F(\gamma).
	\end{align}
	We first show that $F''(\gamma)$ is positive on the interval
    $\left[\frac34,1\right]$. Set
    $$
    G(\gamma):=\log(1+F(\gamma))
    =2\log\Gamma(\gamma+1)-\log\Gamma(2\gamma)-\log(2\gamma-1).
    $$
    Therefore, 
	\begin{equation*}
		G'(\gamma)=2\psi(\gamma+1)-2\psi(2\gamma)
        -\frac{2}{2\gamma-1},
	\end{equation*}
	where \(\psi\) is the digamma function given by \(\psi(z)=\frac{d}{dz}\log(\Gamma(z))=\frac{\Gamma'(z)}{\Gamma(z)}\). 
	Then
	\begin{equation}\label{eq:G''-def}
		G''(\gamma)=2\psi_1(\gamma+1)-4\psi_1(2\gamma)
        +\frac{4}{(2\gamma-1)^2}, 
	\end{equation}
	where \(\psi_1\) is the trigamma function, i.e., \(\psi_1(z)=\psi'(z)\). Since the digamma function \(\psi\) is strictly concave on \((0,\infty)\), 
    \(\psi_1\) is decreasing in the same interval. Therefore,
    $\psi_1(\gamma+1)\ge\psi_1(2)$ and 
    $\psi_1(2\gamma)\le\psi_1(3/2)$. Plugging these bounds in
    \eqref{eq:G''-def} yields 
     $$
     G''(\gamma)\ge 2\psi_1(2)-4\psi_1(3/2)+4,\quad 
     \gamma\in [\tfrac34,1].
     $$
     Using the known values \(\psi_1(2)=\frac{\pi^2}{6}-1\) and \(\psi_1(\tfrac32)=\tfrac{1}{2}\pi^2-4\), we infer that
	\begin{align*}
		G''(\gamma)&\ge
        2\left(\tfrac16\pi^2-1\right)-4\left(\frac{1}{2}\pi^2-4\right)+4\\
		&=18-\frac{5\pi^2}{3}>0 \quad\text{for all }\gamma\in\left[\tfrac34,1\right].
	\end{align*}
	Therefore, \(G\) is convex. This implies that $F(\gamma)=e^{G(\gamma)}-1$ is convex as well and, consequently, $F''(\gamma)$ is positive on $[\tfrac34,1]$.
    Noting that \eqref{eq:F-gamma} yields $f''(q)=4F''(\gamma)$, we conclude that $f''(q)>0$
    for $q\in[0,\tfrac18]$. This implies that
    $$
    f(q)\le\frac{f(1/8)}{1/8}q=
    8\left(\frac{9\Gamma^2(3/4)}{4\sqrt{\pi}}-1\right)q,
    \quad q\in[0,\tfrac18].
    $$
    It remains to notice that the constant on the right hand side is bounded by 7.25.

	\end{proof}

	\bibliographystyle{amsalpha}
	\bibliography{paper_short}
	
\end{document}